\documentclass[12]{amsart}
\usepackage{graphicx}
\usepackage{amscd}
\usepackage{lineno,hyperref}
\usepackage[bottom]{footmisc}
\usepackage{amsmath}
\usepackage{amssymb}
\usepackage{esint}
\usepackage{curves}
\usepackage{array,multirow,graphicx}
\usepackage{tikz}
\usepackage{nicefrac}
\usetikzlibrary{shapes.geometric}
\usetikzlibrary{shapes.multipart}
\RequirePackage{hhline}
\usepackage{enumerate}
\RequirePackage{multirow}
\usetikzlibrary{positioning}

\newtheorem{theorem}{Theorem}[section]

\newtheorem{conjecture}{Conjecture}[section]
\newtheorem{corollary}{Corollary}[section]

\newtheorem{definition}{Definition}[section]
\newtheorem{example}{Example}[section]

\newtheorem{problem}{Problem}[section]
\newtheorem{proposition}{Proposition}[section]
\newtheorem{remark}{Remark}[section]

\begin{document}
\title[ ]{On the Structure of Finite Groups Associated to Regular Non-Centralizer Graph}
\author{Tariq A. Alraqad}
\address{Department of Mathematics, University of Hail, Kingdom of Saudi
Arabia} \email{t.alraqad@uoh.edu.sa}

\author{Hicham Saber}
\address{Department of Mathematics, University of Hail, Kingdom of Saudi
Arabia} \email{hicham.saber7@gmail.com}
\subjclass[2010]{\footnotesize Primary 20B05, Secondary 05C25}
\keywords{Centralizers, Finite Groups, Graph, regular}
\begin{abstract}
The non-centralizer graph of a finite group $G$ is the simple graph $\Upsilon_G$ whose vertices are the elements of $G$ with two vertices $x$ and $y$ are adjacent if their centralizers are distinct. The induced subgroup of $\Upsilon_G$ associated with the vertex set $G\setminus Z(G)$ is called the induced non-centralizer graph of $G$. The notions of non-centralizer and induced non-centralizer graphs were introduced by Tolue in \cite{to15}. A finite group is called regular (resp. induced regular) if its non-centralizer graph (resp. induced non-centralizer graph) is regular. In this paper we study the structure of regular groups as well as induced regular groups. Among the many obtained results,  we prove that if a group $G$ is regular (resp. induced regular) then $G/Z(G)$ as an elementary $2-$group (resp. $p-$group).   
\end{abstract}

\maketitle

\section{Introduction}

For standard terminology and notion in graph theory and group theory, we refer the reader to the text-books of \cite{bo76} and \cite{ro02} respectively. Let $\Upsilon$ be a simple graph. The degree of a vertex $x$ in $\Upsilon$, denoted by $deg(x)$, is the number of vertices adjacent with $x$. $\Upsilon$ is said to be regular if all of its vertices have the same degree. If $deg(x)=n$, we say that $\Upsilon$ is $n-$regular (or regular of degree $n$). 

Throughout this paper, $G$ denotes a finite group. The order of a group $G$ (respectively the order of $x$ in $G$)  is denoted by $|G|$ (respectively $o(x)$. The centralizer of $x$ is $C_G(x)=\{y\in G \mid yx=xy\}$, and the center of $G$ is $Z(G)=\{x\in G\mid xy = yx, \forall y\in G\}$. For an $x$ in $G$, the coset $xZ(G)$ is denoted by $\overline{x}$. The set $Cent(G)=\{C_G(x) \mid x\in G\}$ is the set of distinct centralizers in $G$. 

The notion of non-centralizer graph was introduced by Tolue in \cite{to15}: The non-centralizer graph of a finite group $G$, denoted by $\Upsilon_G$, is the simple graph whose vertices are the elements of $G$ with two vertices $x$ and $y$ are adjacent if $C_G(x)\neq C_G(y)$. The induced subgroup of $\Upsilon_G$ associated with the vertex set $G\setminus Z(G)$ is called the induced non-centralizer graph and denoted by $\Upsilon_{G\setminus Z(G)}$. It is clear that $\Upsilon_G$ is a complete $|Cent(G)|-$partite graph and $\Upsilon_{G\setminus Z(G)}$ is complete $(|Cent(G)|-1)-$partite graph. Classical properties of these graphs such as diameter, girth, domination and chromatic numbers, and independent set were studied in \cite{to15}. Also the author showed that the induced non-centralizer graph and the non-commuting graph associated to an $AC-$group are isomorphic (for information on non-commuting graph see \cite{ab06,ne76}). Interestingly, Tolue proved that  $\Upsilon_G$ is $6$-regular if and only if $G\cong D_8$ or $Q_8$, and  $\Upsilon_G$ is not $n$-regular, for $n=4,5,7,8,11,13$, leading to the following conjecture.

\begin{conjecture}\label{tconj} \cite{to15}
$\Upsilon_G$ is not $p-$regular graph, where $p$ is a prime integer.
\end{conjecture}

In this paper we prove this conjecture by showing that $\Upsilon_G$ in not $n$-regular, if $n$ is a prime power integer. This is  direct consequence of deeper results giving the structure of regular groups such as Theorem \ref{creg} and Theorem \ref{ncen}. 

 In the last few decades, many researches showed interest in characterizing groups by properties of their centralizers such as commutativity  and number of centralizer see for instant \cite{ab07,as00, as000, as06, ba13, be94, fo14, re16, za09}. Our work is no exception of this trend. Indeed, the sets $\beta_G(x)=\{y\in G\mid C_G(y)=C_G(x)\}$, $x\in G$,  as well as the notion of maximal centralizers play key rolls in this paper.  

In the sequel, we say that $G$ is regular (resp. induced regular) if $\Upsilon_G$ is regular (resp. $\Upsilon_{G\setminus Z(G)}$ is regular). In section \ref{reg}, we study the regularity of the non-centralizer graph. We show that if $G$ is regular then $G/Z(G)$ is elementary $2-$group. We also prove that a group is regular if and only if it is the direct product of a regular $2-$group and an abelian group. Following this result, the notion of reduced regular groups is introduced. Moreover, We observe that if $G$ is $n-$regular then every centralizer in $G$ is normal and $n+2\leq |G| \leq 4n/3$. We end this section with a table listing reduced $n-$regular  $2-$groups for all possible values of $n\leq 60$.

Section \ref{ireg} is devoted to the regularity of $\Upsilon_{G\setminus Z(G)}$. Using the concept of maximal centralizers, we obtain many results on the structure of induced regular groups. We prove that if $G$ is induced regular then $G/Z(G)$ is a $p-$group. We also show that a group $G$ is induced regular if and only if it is the direct product of an induced regular $p-$group and an abelian group.  

\section{Regularity of $\Upsilon_G$}\label{reg}

We start this section by listing some known results that will be used later in the sequel. 

\begin{proposition} \label{be0}\cite[Fact 2]{be94}
If $G$ is non-abelian group then $|Cent(G)|\geq 4$.
\end{proposition}

\begin{proposition} \label{be}\cite[Fact 6]{be94}
Let $p$ be a prime. If $G/Z(G)\cong C_p\times C_p$ then $|Cent(G)|=p+2$.
\end{proposition}

The following proposition can be directly obtained from \cite[Proposition 2.2]{ba13} and its proof.
\begin{proposition}\label{ba} Let $G$ be a non-abelian group, such that $[G : Z(G)] = p^3$, where $p$ is the smallest prime dividing $|G|$. If $[G:C_G(x)]=p^2$ for all $x\in G\setminus Z(G)$, then $|Cent(G)|=p^2+p+2$. Otherwise $|Cent(G)|=p^2+2$.
\end{proposition}

For every element $x$ in $G$, define the set $\beta_G(x)=\{y\in G\mid C_G(y)=C_G(x)\}$. These sets form the parts of $\Upsilon_G$. Clearly $deg(x)=|G|-|\beta_G(x)|$ and $\beta_G(e)=Z(G)$. Also  $\overline{x}\subseteq \beta_G(x)$ for all $x\in G$. So $\beta_G(x)$ is a disjoint union of cosets of $Z(G)$. Also we conclude that $|Z(G)|$ divides $deg(x)$ for every $x\in G$, and $|Cent(G)|\leq [G:Z(G)]$. The next proposition shows that $\Upsilon_G$ is regular if and only if for each $x$ in $G$, $\beta_G(x)$ is a coset of $Z(G)$.      

\begin{proposition}\label{ereg1}
Let $G$ be a non-abelian group. Then $G$ is regular if and only if $\beta_G(x)=\overline{x}$ for all $x\in G$.
\end{proposition}

\begin{proof}
We know that $xZ(G)\subseteq\beta_G(x)$. So
\begin{align*}
  G \text{ is regular} &\Longleftrightarrow deg(x)=deg(e)&&\text{for all } x\in G\\
                              & \Longleftrightarrow |G|-|\beta_G(x)|=|G|-|Z(G)|&&\text{for all } x\in G\\
                              &\Longleftrightarrow |\beta_G(x)|=|Z(G)|=|\overline{x}| &&\text{for all } x\in G\\
                              &\Longleftrightarrow  \overline{x}=\beta_G(x)          &&\text{for all } x\in G \\            
  \end{align*} 
\end{proof}	

\begin{corollary}\label{ereg2}
Let $G$ be a non-abelian group. Then $G$ is regular if and only if  $|Cent(G)|= [G:Z(G)]$.
\end{corollary}
\begin{proof}
From Proposition \ref{ereg1}, we see that $G$ is regular if and only if the number of parts of $\Upsilon_G$ is equal to the number of distinct cosets of $Z(G)$. Hence the result. 
\end{proof}

 A group  $G$ is called $p-$group (where $p$ is a prime) if the order of $G$ is a power of $p$. A $p-$group that is the direct product of copies of $C_p$ is called elementary abelian $p-$group. Is it well known in group theory that a group is elementary abelian $2-$group if and only  if the order of each element in $G$ is exactly $2$. For convenience, we call a group $G$ elementary $p-$group if the order of each element in $G$ is exactly $p$.  With this convention, an elementary $p-$group is abelian if and only if it is the direct product of copies of $C_p$.  

\begin{theorem}\label{creg}
Let $G$ be a non-abelian group. If $G$ is regular then $G/Z(G)$ is an elementary abelian 2-group.
\end{theorem}
\begin{proof}
Suppose $ G$ is regular, and let $x\in G$. By Proposition \ref{ereg1}, we have $\beta_G(x^{-1})=\overline{x}^{-1}$. Since $x\in \beta_G(x^{-1})$, we get $x=x^{-1}z$ for some $z\in Z(G)$, and so $x^2=z \in Z(G)$. Therefore, $G/Z(G)$ is an elementary abelian 2-group.
\end{proof}

The converse of Theorem \ref{creg} is not true in general. For example,  the group $G=\langle a,b,c \mid a^4 = b^4 = c^2 = e, ab = ba ,cac^{-1} = a^{-1}, cbc^{-1} = b^{-1} \rangle$ is not regular even though $G/Z(G)\cong C_2\times C_2 \times C_2$. In fact, using GAP we find that for each one of the groups with ID's $[32,27], . . ., [32,35]$,   $G/Z(G)$ is isomorphic to $C_2\times C_2 \times C_2$, yet none of them is  regular. The next two theorems show two cases in which the covers of Theorem \ref{creg} is true.

\begin{theorem}\label{ccreg}
Let $G$ be a non-abelian group. If $G/Z(G)\cong C_2\times C_2$ then $G$ is regular.
\end{theorem}
\begin{proof}
Suppose that $G/Z(G)\cong C_2\times C_2$. Then by Proposition \ref{be},  $|Cent(G)|=4=[G:Z(G)]$. Hence by  Corollary  \ref{ereg2}, $G$ is regular. 
\end{proof}

\begin{example}\label{ex2}
For any positive integer $k\geq 3$, let $M(2^k)$ be the group defined by  \[M(2^k)=\langle a,b\mid a^{2^{k-1}}=b^2=1,\, bab=a^{2^{k-2}+1}\rangle.\] 
We have $Z(M(2^k))=\langle a^2 \rangle$. So $M(2^k)/Z(M(2^k))\cong C_2\times C_2$. Then $ M(2^k)$ is regular of degree $3\cdot 2^{k-2}$.
\end{example}

\begin{theorem}\label{ccreg}
Let $G$ be a non-abelian group such that  $G/Z(G)\cong C_2\times C_2\times C_2$. Then $G$ is regular if and only if  $[G:C_G(x)]= 4$ for all non central elements $x$ of $G$. 
\end{theorem}
\begin{proof}
 Suppose $G$ is regular. Then $|Cent(G)|=[G:Z(G)]=8$. Now for each $x$ in $G\setminus Z(G)$,  $[G:Z(G)]=[G:C_G(x)][C_G(x):Z(G)]$, and so $[G:C_G(x)]\neq 1, 8$. If $[G:C_G(x)]=2$, for some $x\in G $, then Proposition \ref{ba} implies that  $|Cent(G)|=6$, a contradiction. So $[G:C_G(x)]= 4$ for all $x\in G\setminus Z(G)$. Now for the converse, assume that $[G:C_G(x)]= 4$ for all $x\in G\setminus Z(G)$. Then Proposition \ref{ba} implies that $|Cent(G)|=8=[G:Z(G)]$. Hence, by Corollary \ref{ereg2}, $G$ is regular. 
\end{proof}

\begin{theorem}\label{ncen}
Let $G$ be a regular group. Then for each $x \in G$,  $C_G(x)$ is normal in $G$ and $G/C_G(x)$ is isomorphic to a subgroup of $Z(G)$.
\end{theorem}

\begin{proof}
 Suppose $G$ is regular, and let $x \in G$. Since $G/Z(G)$ is abelian, we can associate with every $g \in G$ a unique element $z_{xg}$  in $Z(G)$ such that $xg=gxz_{xg}$. By Proposition \ref{ereg1}, we have that $C_G(xz_{xg})=C_G(x)$. So we obtain that 
\[g^{-1}C_G(x)g=C_G(g^{-1}xg)=C_G(g^{-1}gxz_{xg})=C_G(xz_{xg})=C_G(x).\] Hence $C_G(x)$ is normal in $G$. Also the mapping $\phi_x:G\longrightarrow Z(G)$ defined by $\phi_x(g)=z_{xg}$  is a homomorphism with $Ker(\phi_x)=C_G(x)$. Thus $G/C_G(x)\cong Img(\phi_x)$.   
\end{proof}

The following corollary proves Conjecture \ref{tconj}.

\begin{corollary}\label{preg}
Let  $G$ be a non-abelian group. Then $\Upsilon_G$ is not $n-$regular, where $n$ is a prime power integer.
\end{corollary}
\begin{proof}
We know that  $n=|G|-|Z(G)|=|Z(G)|([G:Z(G)]-1)$. By Theorem \ref{creg} and Theorem \ref{ncen} we get that $[G:Z(G)]-1$ is odd and $|Z(G)|$ is even. Hence $n$ can not be a power of prime.    
\end{proof}

\begin{theorem}\label{bound}
If $G$ is a non-abelian group such that $\Upsilon_G$ is $n-$regular, then $n$ is even, $|G|\equiv 0 \pmod{8}$, and $n+2\leq |G| \leq 4n/3$.
\end{theorem}
\begin{proof}
Corollary \ref{ereg2} yields $[G:Z(G)]$ is dividable by $4$, and Theorem \ref{ncen} implies that $|Z(G)|$ is even. Hence $|G|$ is divisible by $8$. Since $|G|=\frac{[G:Z(G)]n}{[G:Z(G)]-1}$, form Proposition \ref{be0} we get $[G:Z(G)]\geq 4$, and so we obtain that $|G| \leq 4n/3$. In addition,  $2\leq |Z(G)|=|G|-n$. Therefor  $n+2\leq|G| \leq 4n/3$. 
\end{proof}

\begin{theorem}\label{big} Let $k\geq3$ and $t\geq1$ be integers and let $G$ be a group of order $2^k(2t-1)$. Then   $G$ is regular if and only if $G\cong H\times A$ where $A$ is an abelian group of order $2k-1$, $H$ is a regular $2-$group of order $2^k$. 
\end{theorem}
\begin{proof}
Suppose $G$ is regular. Then $G/Z(G)$ is a $2-$group, which yields  $|Z(G)|=2^s(2t-1)$ for some $1\leq s \leq k$. Let $A$ be the subgroup of $Z(G)$ of order $2t-1$ and let $H$ be a $2-$sylow subgroup of $G$. Now $HA=AH$. So $HA\leq G$. Also, we have $H\cap A=\{e\}$ (this is because $H$ is a $2-$ group and $|A|$ is odd). Thus $|HA|=|H||A|=|G|$, and so $G=HA$. Moreover, since $ha=ah$ for all $a\in A$ and $h\in H$, we have $HA\cong H\times A$. It remains to show that $H$ is regular. Since $A\leq Z(G)$, we obtain that for any two elements $x,y\in H$,  $C_G(x)=C_G(y)$ if and only if $C_H(x)=C_H(y)$. So, for each $x\in H$,  $\beta_H(x)=\beta_G(x)\cap H$. Now $Z(G)=Z(H)A$, also by Proposition \ref{ereg1}, we have $\beta_G(x)=\overline{x}$. Therefore
\[\beta_H(x)= \beta_G(x)\cap H=\overline{x}\cap H=xZ(H)A\cap H= xZ(H).\] 
Then  Proposition \ref{ereg1} yields regularity of $H$.

Now we prove the converse. Assume that $G=H\times A$ where $A$ is an abelian and $H$ is regular $2-$group. We have $Z(G)=Z(H)\times Z(A)=Z(H)\times A$. Also, for each $(h,a)\in G$, $C_G(h,a)=C_H(h)\times C_A(a)=C_H(h)\times A$. Combining this we obtain 
\[\beta_G(h,a)=\beta_H(h)\times A=hZ(H)\times aA=(h,a)(Z(H)\times A)=(h,a)Z(G).\] 
Therefore, by Proposition \ref{ereg1} $G$, is regular.
\end{proof}

\begin{example} Let $A$ be an abelian group of order $k\geq1$ and let  $G=D_8\times A$ or  $G=Q_8\times A$. Then $\Upsilon_G$ is $K_{2k,2k,2k,2k}$ which is $6k-$regular. \end{example}

We get the following two remarks from Theorem \ref{big} and its proof.
\begin{remark}
The sylow subgroup $H$ is normal in $G$, and the abelian subgroup $A$ is the direct product of all other sylow subgroups. So all sylow subgroups of $G$ are normal, and $G$ is isomorphic to the direct product of its sylow subgroups. 
\end{remark}

\begin{remark}
Characterizing regular groups reduces to studying regular $2-$groups that are not the direct product of a regular group and an abelian group. 
\end{remark}

\begin{definition}
A regular $2-$group is called reduced regular if it is not the direct product of a regular group and an abelian group.   
\end{definition}

Using Theorem \ref{bound}, we can find, for a fixed value of $n$, all $n-$regular and reduced $n-$regular groups. For example, for $n=6$, Theorem \ref{bound} implies that $|G|$ must be $8$. Among the groups of order $8$ we find that  $Q_8$, $D_8$ are the only $6-$regular groups and both of them are  reduced regular. For the case $n=10$, we obtain that $|G|=12$. But also $|G|$ must be divisible by $8$.  So there no  $10-$regular groups. For $n=12$ we find that $|G|=16$. After investigating the groups of order $16$, we find that six of them are $10-$regular, and four out of these six are reduced $10-$regular. Table \ref{tab1} lists all reduced $n-$regular $2-$groups with  $n\leq 60$. 

\begin{table}[h]
\caption{}
\centering 
\begin{tabular}{|c|l|}
\hline
$n$&  Reduced $n-$Regular $2-$Groups \\
\hhline{|=|=|}

6&  $Q_8$, $D_8$\\\hline
12& [16,3], [16,4], [16,6], [16,13]\\\hline
24& [32,2], [32,4], [32,5], [32,12], [32,17],  [32,24], [32,38]\\\hline
30&[32,49], [32,50] \\\hline
48&[64,3], [64,17], [64,27], [64,29], [64,44],\\
  &[64,51], [64,57], [64, 86], [64, 112], [64, 185] \\\hline
56&[64, 73], [64, 74], [64, 75], [64, 76], [64, 77] \\\hline
&[64, 78], [64, 79], [64, 80], [64, 81],[64, 82]\\\hline
60&[64, 199], [64, 200], [64, 201], [64, 226], [64, 227],\\
  &[64, 228], [64, 229], [64, 230], [64, 231], [64, 232],\\
	&[64, 233], [64, 234], [64, 235], [64, 236], [64, 237], \\
	&[64, 238], [64, 239], [64, 240], [64, 249], [64, 266] \\\hline
\end{tabular}
\label{tab1}
\end{table}

\section{Regularity of $\Upsilon_{G\setminus Z(G)}$}\label{ireg}

\begin{definition}\label{mc}
Let $G$ be a non-abelian group. A centralizer $C_G(x)$ is called maximal if it is not contained in any other proper centralizer. 
\end{definition} 

\begin{proposition}\label{lg}
Let $G$ be a group. If  $C_G(x)$ is a maximal centralizer in $G$, then $\beta_G(x)\cup Z(G)$ is a subgroup of $G$.
\end{proposition}

\begin{proof}
Since $C_G(y)=C_G(y^{-1})$, we get that $y\in B_G(x)\cup Z(G)$ if and only if  $y^{-1}\in \beta_G(x)\cup Z(G)$. Now, let $y,w\in \beta_G(x)$. Since $C_G(y)=C_G(x)=C_G(w)$, we get $C_G(x)\subseteq C_G(yw)$. So by maximality of $C_G(x)$, we obtain that $C_G(x)=C_G(yw)$ or $C_G(yw)=G$. In both cases $yw$ belongs to $\beta_G(x)\cup Z(G)$. Therefore  $\beta_G(x)\cup Z(G)$ is a subgroup.
\end{proof}

\begin{proposition}\label{lg1}
Let $G$ be an induced regular group, and let  $C_G(x)$ be a maximal centralizer such that $C_G(x)\neq \beta_G(x)\cup Z(G)$. Then there is an element $y$ in $C_G(x)\setminus \beta_G(x)$ such that $o(\overline{y})$ is prime.
\end{proposition}

\begin{proof}
Assume that for each $y\in C_G(x)\setminus \beta_G(x)$, $o(\overline{y})$ is not prime. Then for each $w\in C_G(x)\setminus(\beta_G(x)\cup Z(G))$, $C_G(w)\subseteq C_G(x)$. Indeed, take a prime divisor $p$ of $o(\overline{w})$, then $w^{\frac{o(\overline{w})}{p}}\in \beta_G(x)$. Hence $C_G(w)\subseteq C_G(w^{\frac{o(\overline{w})}{p}})= C_G(x)$. Now let $y\in C_G(x)\setminus \beta_G(x)\cup Z(G)$. We claim that $y\beta_G(x)\subseteq \beta_G(y)$.  By Proposition \ref{lg},  $\beta_G(x)\cup Z(G)$ is a group. This implies that for all $g\in \beta_G(x)$, $yg\notin \beta_G(x)\cup Z(G)$. Now let $g\in \beta_G(x)$, we want to show that $C_G(gy)=C_G(y)$. Clearly $C_G(g)\cap C_G(y)\subseteq C_G(gy)$, and so $C_G(y)=C_G(x)\cap C_G(y)=C_G(g)\cap C_G(y)\subseteq C_G(gy)$. For the other inclusion, let $h\in C_G(gy)$, then $hgy=gyh$. But also $h$ is in $C_G(x)=C_G(g)$, and so $ghy=gyh$, which yields $h\in C_G(y)$.  This proves our claim. Since $G$ is induced regular group, we get that $|y\beta_G(x)|=|\beta_G(y)|$, and so $y\beta_G(x)=\beta_G(y)$. This implies that $1\in \beta_G(x)$, a contradiction. Therefore, there must be  $y \in C_G(x)\setminus \beta_G(x)$ such that $o(\overline{y})$ is prime.  
\end{proof}

\begin{proposition}\label{lg2}
Let $G$ be an induced regular group, and let $C_G(x)$ be a maximal centralizer in $G$. If there exists an element $y$ in $C_G(x)\setminus \beta_G(x)$ such that $o(\overline{y})$ is a prime $p\neq2$, then $\dfrac{\beta_G(x)\cup Z(G)}{Z(G)}$ is an elementary $p-$ group.
\end{proposition}
\begin{proof}
Let $H_x=\beta_G(x)\cup Z(G)$ and let $B=\{w\in \beta_G(x) \mid o(\overline{w})\neq p\}$. For  $w\in B$, $C_G(w)\cap C_G(y)\subseteq C_G(wy)$. Also , $C_G(wy)\subseteq C_G((wy)^p)=C_G(w^p)=C_G(w)$. Then $C_G(w)\cap C_G(y)=C_G(wy)$. Hence we obtain that $yB\subseteq \beta_G(w_0y)$ for some $w_0\in B$.  Let $T=\left\langle \{w\in H_{x}\mid o(\overline{w})=p\}\right\rangle$. Then $T/Z(G)$ is an elementary $p-$subgroup of $H_{x}/Z(G)$.  We have $y(H_{x}\setminus T) \subseteq \beta_G(w_0y)$ and  $y^{-1}(H_{x}\setminus T) \subseteq \beta_G(w_0y)$. If $y(H_{x}\setminus T)$ and $y^{-1}(H_{x}\setminus T)$ are not disjoint then $\overline{y}\overline{u}=\overline{y}^{-1}\overline{v}$ for some $u,v\in H_{x}\setminus T$. This implies that $y^2\in \overline{v}\,\overline{u}^{-1}\subseteq \beta_G(x)\cup Z(G)$. But also, since $y^p\in Z(G)$ and $p\neq2$, we get that $y\in \beta_G(x)$, a contradiction. So $y(H_{x}\setminus T)$ and $y^{-1}(H_{x}\setminus T)$ are disjoint. This implies that 
\[2|y(H_{x}\setminus T)|=|y(H_{x}\setminus T)|+|y^{-1}(H_{x}\setminus T)|\leq  |\beta_G(w_0y)| <|H_{x}|.\] 
So $2(|H_{x}|-|T|)<|H_{x}|$, which implies that $|H_{x}|<2|T|$. Thus $H_{x}=T$, and hence $H_{x}/Z(G)$ is an elementary $p-$group.
\end{proof}

\begin{theorem}\label{mg}
 If $G$ be a non-abelian induced regular group then  $G/Z(G)$ is a $p$-group.
\end{theorem}

\begin{proof}
Let $G$ be a non-abelian induced regular group. If $C_G(x)=\beta_G(x)\cup Z(G)$, for all $x\in G$, then $G$ is an AC-group. Thus by \cite[Theorem 2.11]{to15} and \cite[Proposition 2.6]{ab06}, we have $G=P\times A$, where $P$ is a $p-$group and $A$ is an abelian group. Hence, $\dfrac{G}{Z(G)}=\dfrac{P\times A}{Z(P)\times A}\cong \dfrac{P}{Z(P)}$, and so it is a $p-$group.
Now suppose that $G$ is not an AC-group. Let $C_G(x_0)$ be a non-abelian maximal centralizer of $G$ and let $H_{x_0}=\beta_G(x_0)\cup Z(G)$. By Proposition \ref{lg1}, there exists an element $y$ in $C_G(x_0)\setminus H_{x_0}$ such that $o(\overline{y})$ is prime $p$. The next goal is to show the following claim.

Claim: $H_{x_0}/Z(G)$ is a $p-$group.

If there is an element $y$ in $C_G(x_0)\setminus H_{x_0}$ such that $o(\overline{y})=p\neq2$, then by Proposition \ref{lg2}, $H_{x_0}/Z(G)$ is an elementary $p-$group.  Now suppose that for each element $y$ in $C_G(x_0)\setminus H_x$, $o(\overline{y})$ is either $2$ or not prime. In this case we show that $H_{x_0}/Z(G)$ is a $2-$group. This is achieved in three steps .

Step 1: We show that $C_G(x_0)/H_{x_0}$ is a $2-$group.  For simplicity, we will denote the coset $yH_{x_0}$ in $C_G(x_0)$ by $\tilde{y}$. Assume there is $y\in C_G(x_0)$ such that $o(\tilde{y})$ is divisible by a prime $p\neq2$, and let $y_0=y^{\frac{o(\tilde{y})}{p}}$. Clearly  $o(\tilde{y}_0)=p$ and $x_0y_0\notin H_{x_0}$. Now $y^p_0$ and $(x_0y_0)^p$ belong to $H_{x_0}$. But since $o(\overline{y}_0)$ and $o(\overline{x_0}\overline{y}_0)$ are either $2$ or not prime, we have $y^p_0$ and $(x_0y_0)^p$ belong to $\beta_G(x_0)$. Therefor  $C_G(x_0y_0) \cup C_G(y_0) \subseteq C_G((x_0y_0)^p)\cup  C_G((y_0)^p)=C_G(x_0)$. Using the same method as in the proof of Proposition \ref{lg1} one can easily show that $C_G(wy_0)=C_G(x_0)\cap C_G(y_0)=C_G(y_0)$ for all $w\in \beta_G(x)$. Thus we obtain that $y_0\beta_G(x_0)\subseteq \beta(y_0)$. From the fact that $y_0\notin y_0\beta_G(x_0)$, we get that $|\beta_G(y_0)|> |y_0\beta_G(x_0)|=|\beta_G(x_0)|$, a contradiction. Therefore $C_G(x_0)/H_{x_0}$ is a $2-$group.

Step 2: We show that for all $y\notin H_{x_0}$, $o(\overline{y})$ is a power of $2$. Let $y\notin H_{x_0}$, and suppose that $o(\overline{y})=2^k(2t-1)$. If $y^{2t-1}$ belongs to $\beta_G(x_0)$, then $(\tilde{y})^{2t-1}=\tilde{1}$ in $C_G(x_0)/H_{x_0}$, and so by Step 1, we get that $2$ divides $2t-1$, a contradiction. Thus $y^{2t-1}\notin \beta_G(x)$. Now suppose that $p$ is a prime divisor of $2t-1$ and let  $y_p=y^{\frac{2t-1}{p}}$. Then $o(\overline{y}_p)=p$, and so $y_p\in \beta_G(x_0)$. This implies that $2$ divides $(2t-1)/p$, a contradiction. Thus $2t-1=1$, and hence $o(\overline{y})$ is a power of $2$.

Step 3: We show that $\frac{H_{x_0}}{Z(G)}$ is a $2-$group. Assume that there is an element $w\in H_{x_0}$ such that $o(\overline{w})=p\neq2$. By Proposition \ref{lg1}, there is $y\in C_G(x_0)\setminus H_{x_0}$ such that $o(\overline{y})=2$. Then we have $wy\notin H_{x_0}$ and $o(\overline{wy})=2p$, which contradicts the result in Step 2. Hence $\frac{H_{x_0}}{Z(G)}$ is a $2-$group. This completes the proof of the claim.

Now let $C_G(x)$ be an arbitrary maximal centralizer in $G$, and let $H_x=\beta_G(x)\cup Z(G)$. Then from the regularity of $\Upsilon_{G\setminus Z(G)}$ we obtain that $|H_x|=|H_{x_0}|$. So $|H_x|$ is a $p-$group. Now we come to the last step in the proof which is showing that $G/Z(G)$ is a $p-$group. Assume that $[G:Z(G)]$ is divisible by a prime $q\neq p$. Let $y\in G$ such that $o(\overline{y})=q$, and let $C_G(x)$  be maximal centralizer that contains $C_G(y)$. Again following the same argument as in the proof of Proposition \ref{lg1} one can show that  $y\beta_G(x)\subseteq \beta_G(y)$. But since $y\notin y\beta_G(x)$ we obtain that $|\beta_G(x)|=|y\beta_G(x)|<|\beta_G(y)|$, which contradicts the fact that $G$ is induced regular. Therefore $G/Z(G)$ is a $p-$group.      
\end{proof}

\begin{proposition} Let $G$ be an induced regular group. If $[G:Z(G)]=p^q$ where $p$ and $q$ are primes then $G/Z(G)$ is an elementary $p-$group. Moreover, for each $x\in G$, $|\beta_G(x)|=(p-1)|Z(G)|$. 
\end{proposition}
\begin{proof}
Suppose $[G:Z(G)]=p^q$, for some $q\geq 2$. Consider a maximal centralizer $C_G(x)$ in $G$. Then $|\beta_G(x)|/|Z(G)|=p^s-1$ for some $1\leq s\leq q-1$ and $(|Cent(G)|-1)(p^s-1)=p^q-1$, hence $p^s-1$ divides $p^q-1$. This implies that $s$ divides $q$. Thus, if $q$ is prime, then $s=1$, which completes the proof.
\end{proof}

\begin{corollary}\label{cmg}
Let $G$ be an induced regular group of odd order. If $C_G(x)\neq\beta_G(x)\cup Z(G)$, for all $x\in G$  then $G/Z(G)$ is an elementary $p$-group.
\end{corollary}
\begin{proof}
This result follows directly from Proposition \ref{lg2}.
\end{proof}

\begin{proposition}\label{pp}
 Let $G$ be a non-abelian group. If there is a prime $p$ such that $G/Z(G)\cong C_p\times C_p$, then $G$ is induced regular.
\end{proposition}

\begin{proof}
Suppose that $G/Z(G)\cong C_p\times C_p$. Then by Proposition \ref{be}, $\Upsilon_{G\setminus Z(G)}$  has $p+1$ parts. Let $\beta_G(x_1),\cdots, \beta_G(x_{p+1})$ be the parts $\Upsilon_{G\setminus Z(G)}$ . Then 
\[\frac{1}{|Z(G)|}\sum\limits_{i=1}^{p+1}{|\beta_G(x_i)|}=[G:Z(G)]-1=p^2-1.\] 
On the other hand, $o(\overline{x}_i)=p$ for all $i$. Then for each $i\in\{1,\cdots,p+1\}$, $\bigcup\limits_{j=1}^{p-1}\overline{x}_i^{j}\subseteq \beta_G(x_i)$, which implies that $\frac{|\beta_G(x_i)|}{|Z(G)|}\geq (p-1)$, for all $i$. If we assume that  $\frac{|\beta_G(x_i)|}{|Z(G)|}> (p-1)$, for some $i$, then we obtain that 
\[p^2-1=\frac{1}{|Z(G)|}\sum\limits_{i=1}^{p+1}{|\beta_G(x_i)|}> (p+1)(p-1)=p^2-1,\] 
a contradiction. Thus $\frac{|\beta_G(x_i)|}{|Z(G)|}=(p-1)$, for all $i$, and hence $\Upsilon_{G\setminus Z(G)}$ is regular.
\end{proof}

\begin{theorem}\label{big1} A group $G$ is induced regular if and only if $G\cong H\times A$ where $A$ is an abelian group, and $H$ is an induced regular $p-$group for some prime $p$. 
\end{theorem}
\begin{proof}
Suppose $G$ is induced regular. Then $G/Z(G)$ is a $p-$group. Let $|G|=p^km$, where $(m,p)=1$. Then $|Z(G)|=p^sm$ for some $1\leq s \leq k$. Let $A$ be the subgroup of $Z(G)$ of order $m$ and let $H$ be a $p-$sylow subgroup of $G$. Now $HA=AH$. So $HA\leq G$. Also, we have $H\cap A=\{e\}$ (because $(|H|,|A|)=1$), and hence $|HA|=|H||A|=|G|$. Therefor $G=HA$. Moreover, since $ha=ah$ for all $a\in A$ and $h\in H$, we have $HA\cong H\times A$. It remains to show that $H$ is induced regular. Fix $x$ in $H\setminus Z(H)$ and let $s=|\beta_G(x)|/|Z(G)|$. Since $A\leq Z(G)$, we obtain that $y$ belongs to $\beta_G(x)$ if and only if $C_H(x)=C_H(y)$. So $\beta_H(x)=\beta_G(x)\cap H$. Also since $Z(G)=Z(H)A$, there are $y_1,\cdots, y_s$ in $H$ such that  $\beta_G(x)= \bigsqcup\limits_{i=1}^{s}y_i Z(G)$. So \[\beta_H(x)= \beta_G(x)\cap H=\left(\bigsqcup\limits_{i=1}^{s}y_iZ(G)\right)\cap H=\bigsqcup\limits_{i=1}^{s}y_i(Z(G)\cap H)=\bigsqcup\limits_{i=1}^{s}y_iZ(H).\] Thus $|\beta_H(x)|=\dfrac{|\beta_G(x)|}{|Z(G)|}|Z(H)|$. Therefore $H$ is induced regular.

Now we prove the converse. Assume that $G=H\times A$ where $A$ is an abelian group, $H$ is an induced regular group of order $p^k$. Now $Z(G)=Z(H)\times Z(A)=Z(H)\times A$. Also, for each $(h,a)\in G$, $C_G(h,a)=C_H(h)\times C_A(a)=C_H(h)\times A$. So $\beta_G(h,a)=\beta_H(h)\times A$, for all $(h,a)\in G$. Therefore $G$ is induced regular.
\end{proof}

It is worth mentioning that there are several induced regular groups $G$ for which $G/Z(G)$ is not abelian such as the groups with GAP ID's $[243,2]$ to $[243,9]$.  On the other hand, we checked many groups and we could not find one induced regular group $G$ for which $G/Z(G)$ is not elementary. Based on these remarks and the results obtained in this section we make the following conjecture. 

\begin{conjecture}\label{lco}
If $G$ is induced regular group then $G/Z(G)$ is an elementary $p-$group.
\end{conjecture}

The following problem state the remaining part needed to completely prove conjecture \ref{lco} .

\begin{problem}
Let $G$ is induced regular group, and let $C_G(x)$ be a maximal centralizer such that $C_G(x)=\beta_G(x)\cup Z(G)$ or 
$(\beta_G(x)\cup Z(G))/Z(G)$ is $2-$group. Then $\beta_G(x)\cup Z(G)$ is elementary $p-$group.
\end{problem}

\end{document}